\documentclass[a4wide,12pt]{article}

\usepackage[pdftex]{graphicx}
\usepackage{amsmath,amssymb,amsfonts}
\usepackage[textsize=small,textwidth=1.3in]{todonotes}
\usepackage[english]{babel}
\usepackage{enumerate}
\usepackage{multicol}

\newtheorem{prop}{Proposition}[section]
\newtheorem{thm}[prop]{Theorem}
\newtheorem{lemma}[prop]{Lemma}
\newtheorem{cor}[prop]{Corollary}
\newtheorem{Defi}[prop]{Definition}
\newtheorem{Rem}[prop]{Remark}
\newtheorem{Exam}[prop]{Example}
\newtheorem{Exer}[prop]{Exercise}
\newtheorem{Expe}[prop]{Experiment}
\newtheorem{Cons}[prop]{Construction}

\newtheorem{Alg}[prop]{Algorithm}
\newtheorem{Prob}[prop]{Research Problem}
\newenvironment{defi}{\begin{Defi} \rm}{\end{Defi}}
\newenvironment{rem}{\begin{Rem} \rm}{\end{Rem}}
\newenvironment{ex}{\begin{Exam} \rm}{\end{Exam}}

\newenvironment{proof}{\noindent{\bf Proof.}\ }{\hspace*{\fill}$\diamond$\medskip\par}

\newcommand{\ff}{{\mathbb F}}

\newcommand{\fq}{{\mathbb F}_q}
\newcommand{\fqm}{{\mathbb F}_{q^m}}

\newcommand{\ba}{{\bf a}}
\newcommand{\bb}{{\bf b}}
\newcommand{\bc}{{\bf c}}
\newcommand{\bd}{{\bf d}}
\newcommand{\be}{{\bf e}}

\newcommand{\br}{{\bf r}}
\newcommand{\bs}{{\bf s}}
\newcommand{\bu}{{\bf u}}
\newcommand{\bv}{{\bf v}}
\newcommand{\bx}{{\bf x}}

\newcommand{\scrx}{{\cal X}}

\title{A characterization of MDS codes \\
that have an  error correcting pair}

\author{Irene M\'arquez-Corbella
\footnote{SECRET Project-Team - INRIA, Paris-Rocquencourt, B.P. 105, 78153 Le Chesnay Cedex France, E-mail: irene.marquez-corbella@inria.fr}
    \     and Ruud Pellikaan
\footnote{Department of Mathematics and Computing Science, Eindhoven University of Technology, P.O. Box 513, 5600 MB  Eindhoven, The Netherlands. E-mail: g.r.pellikaan@tue.nl} }

\begin{document}

\maketitle
\begin{abstract}
Error-correcting pairs were introduced in 1988 in the pre-print \cite{pellikaan:1988} that appeared in \cite{pellikaan:1992},
and were found independently in \cite{koetter:1992}, as a general algebraic method of decoding linear codes.
These pairs exist for several classes of codes.
However little or no study has been made for characterizing those codes. This article is an attempt to fill the vacuum left by the literature concerning this subject. Since every linear code is contained in  an MDS code of the same minimum distance over some finite field extension, see \cite{pellikaan:1996}, we have focused our study on the class of MDS codes.
Our main result states that an MDS code of minimum distance $2t+1$ has a $t$-ECP if and only if it is a generalized Reed-Solomon code.
A second proof is given using recent results \cite{mirandola:2015,randriambololona:2013} on the Schur product of codes.
\end{abstract}

\section{Introduction}
\label{Introduction}
Error-correcting pairs were introduced in \cite{pellikaan:1988,pellikaan:1992}, and independently in \cite{koetter:1992}, as a general algebraic method of decoding linear codes.
These pairs exist for several classes of codes such as for generalized Reed-Solomon, cyclic, alternant and algebraic geometry codes \cite{duursma:1993a,duursma:1994,koetter:1992,koetter:1996,pellikaan:1989,pellikaan:1992,pellikaan:1996}.
The aim of this paper is to characterize those  MDS codes that have a $t$-error correcting pair.
This was shown for $t\leq 2$ in \cite{pellikaan:1996}.

Section \ref{MDS:section} gives the background  of MDS codes.
Generalized Reed-Solomon codes and an equivalent way to describe such a code as a projective system on a rational normal curve in projective space is reviewed in Section \ref{GRS:section}. We give here also a survey of well-known results related to generalized Reed-Solomon codes which are used to set the notation and to recall some properties that are relevant for the proof of the main result.
Moreover, a classical result is stated: a rational rational curve in projective $r$ space is uniquely determined by $n$ of its points in case $n \geq r+2$. This classical result will be vital in our main result.

For further details on the notion of an error correcting pair see Section \ref{ECP:section} were we formally review this definition, detailing the state-of-art and the existence of error correcting pairs for some families of codes.

Section \ref{Puncturing:Section} gives a brief exposition of classical methods of constructing a shorter code out of a given one: the process of puncturing and shortening a code. In particular we will be concerned with the case of these operations for MDS codes.

Finally, in Section \ref{MT:section} we present the main result of this paper that states that every MDS code with minimum distance $2t+1$ that has a $t$-ECP belongs to the class of generalized Reed-Solomon codes.
In Section \ref{Second:Section} we extend recent results on the Schur product of codes
\cite{randriambololona:2013,mirandola:2015} to give an independent proof of our main result.

\subsection{Notation}
\label{Notation:section}

By $\mathbb F_q$ where $q$ is a primer power, we denote a finite field with $q$ elements.
The projective line over the finite field $\mathbb F_q$, denoted by $\mathbb P^1(\mathbb F_q)$, is the set $\mathbb F_q\cup \{ \infty\}$, and $\mathbb F_q^*$ denote the set of units of $\mathbb F_q$.

An $[n,k]$ linear code $C$ over $\mathbb F_q$ is a $k$-dimensional subspace of $\mathbb F_q^n$.
We will denote the length of $C$ by $n(C)$, its dimension by $k(C)$ and its minimum distance, $d(C)$.

Given two elements $\mathbf a$ and $\mathbf b$ on $\mathbb F_q^n$ the \emph{star multiplication} is defined by coordinatewise multiplication, that is
$\mathbf a * \mathbf b = (a_1b_1, \ldots, a_nb_n)$.
Then, $A*B$ is the code in $\fq ^n$ generated by $\left\{ \mathbf a*\mathbf b \mid \mathbf a \in A \hbox{ and }\mathbf b \in B\right\}$.

Note that in this paper  $A*B$ is not the set $\left\{ \mathbf a*\mathbf b \mid \mathbf a \in A \hbox{ and }\mathbf b \in B\right\}$
as in \cite{pellikaan:1992}, but the space generated by that set.

The \emph{standard inner multiplication} is defined by $\mathbf a\cdot \mathbf b = \sum_{i=1}^n a_ib_i$ of $\mathbf a$ and $\mathbf b$ on $\mathbb F_q^n$.
 Let $A$ and $B$ be two codes in $\mathbb F_q^n$.
Now $A \perp B$ if and only if $\mathbf a \cdot \mathbf b = 0$ for all $\mathbf a \in A$ and $\mathbf b \in B$.

\section{MDS codes}
\label{MDS:section}
 The \emph{Singleton bound} provides an upper bound for the minimum distance given by
$d(C) \leq n(C) - k(C) + 1$.
If equality holds then the minimum distance of $C$ achieves its maximum possible value, in this case we say that $C$ is a \emph{maximum distance separable code} or a MDS code, for short.

Trivial examples of MDS codes are the zero code of length $n$ which has parameters $[n,0,n+1]$ by definition and its dual, that is the whole space $\mathbb F_q^n$ which has parameters $[n,n,1]$. Moreover the $[n,1,n]$-repetition code over $\mathbb F_q$ and its dual are both MDS. Indeed the dual of an $[n,k]$ MDS code is an $[n,n-k]$ MDS code. Other well-known examples of MDS codes are (extended/generalized) Reed-Solomon codes \cite{reed:1960} which can be seen as algebraic geometry codes on the projective line, that is on algebraic curves of genus zero.
This family, that will be studied in detail in Section \ref{GRS:section}, is  very useful since  it is fundamental in several technological applications ranging from computer drives, CD or DVD players to digital imaging, see \cite{wicker:1995}.

Singleton introduced MDS codes in $1964$ \cite{singleton:1964} but curiously his bound was already discovered in $1952$ by Bush \cite{bush:1952}
in the context of orthogonal arrays. McWilliam and Sloane in \cite{macwilliams:1977} describe this class of codes as
\emph{'One of the most fascinating chapters in all of coding theory'}.
However these codes are not only of interest to coding theory  but also to a wide range of different areas such as projective geometry over finite fields and combinatorics, see \cite{alderson:2007, blokhuis:1990, bruen:1988, thas:1998} and more recently in the theory of quantum error correcting codes, see \cite{calderbank:1996}.

We have defined MDS codes as the codes that attain the maximum error detection/correction capability, but this is not the only property that makes them special. In the next theorem we have collected some properties characterizing MDS codes.
This result is presented without proof since it is well known. For more details we refer the reader to \cite{huffman:2003,macwilliams:1977}.

\begin{thm}
Let $C$ be an $[n,k]$ code over $\mathbb F_q$. Then, the following statements are equivalent:
\begin{enumerate}
\itemsep0em
\item $C$ is MDS.
\item $C^{\perp}$ is MDS.
\item Every $k$-tuple of columns of a generator matrix of $C$ is independent. 
\item Every set of $k$ coordinates form an information set.
\item Every $n-k$-tuple of columns of a parity check matrix of $C$ is independent.
\item $C$ is generated systematically by a matrix $G=\left(\begin{array}{cc}I_k & A\end{array}\right)$
where every square submatrix of $A$ is nonsingular and up to linear isometry $A$ is of the form:
$$
A = \left(\begin{array}{c|c}
    1 & \mathbf 1 \\
    \hline
    \mathbf 1^T & *
    \end{array} \right)\in \mathbb F_q^{k\times (n-k)}.
$$
\end{enumerate}
\end{thm}

Other properties are for example the statement that every shortened MDS code is again MDS or the uniqueness of the weight distribution of a MDS code which up to some explicit exceptions, it contains all weights in the range $[n-k+1, n]$, see \cite{ezerman:2011}.

In general, the maximum length of MDS codes with given dimension and defining field is unknown. Indeed the main conjecture on MDS codes states that if there is an $[n,k]$ MDS code over $\mathbb F_q$ with $2\leq k \leq n-2$. Then, $n\leq q+1$ unless $q$ is even and $k=3$ or $k=q-1$ in which case $n\leq q+2$. Several special cases are solved. Most recently the case $q$ a prime is settled \cite{ball:2012}. But this question is a large unsolved problem.
Wicker \cite{wicker:1995} describes an analogous question on arcs in the projective space which has become one of the most interesting problems in projective geometry over Galois fields. However at least asymptotically, that is for large $q$ the class of Reed-Solomon codes are optimal in the family of MDS codes, see \cite{bruen:1988}. Moreover for all $k\in \{ 1, \ldots, n\}$ there is a $[q-1,k]$ Reed-Solomon code over $\mathbb F_q$ and we can add an overall parity check obtaining a $[q,k]$ MDS code known as \emph{Extended RS} code. Thus a natural approach to find long MDS codes is to lengthen a fixed code while preserving the MDS property, but in general no more than a finite parity check can be added keeping the desired property. Notice also that in all the known cases when an $[n,k]$ MDS code exists, a GRS code with the same parameters also exists.

\section{Generalized Reed-Solomon codes}
\label{GRS:section}

We write $L_k$ for the set of polynomials in $\mathbb F_q[X]$ of degree at most $k-1$, that is
$L_k = \left\{ f\in \mathbb F_q[X] \mid \deg(f(X)) \leq k-1\right\}$.

Let $\mathbf a = (a_1, \ldots, a_n)\in \mathbb F_q^n$ be an $n$-tuple of mutually distinct elements of $\mathbb F_q$, and let $\mathbf b=(b_1, \ldots, b_n)$ be an $n$-tuple of nonzero elements of $\mathbb F_q$. Note that the definition of $\mathbf a$ requires the length $n$ to be at most $q$.

Consider the \emph{evaluation map}
$\begin{array}{cccc}
\mathrm{ev}_{\mathbf a, \mathbf b}: & L_k & \longrightarrow & \mathbb F_q^n
\end{array}$
where the elements $\mathrm{ev}_{\mathbf a, \mathbf b}(f) = \mathbf b * f(\mathbf a)$ arises from evaluating any polynomial $f(X)\in L_k$ at the points of $\mathbf a$ and  scaling by $\mathbf b$. Recall that the evaluation of a polynomial $f(X)=f_0 + f_1 X+ \ldots f_{k-1}X^{k-1}\in L_k$ at the points of $\mathbf a$ is equivalent to multiply the coefficient vector of the polynomial $f(X)$ by the $k\times n$ Vandermonde matrix
$G_{\mathbf a}$ with entry $ a_j^{i-1}$ at the $i$-th row and the $j$-th column. Therefore,
$$\mathbf b* f(\mathbf a) =
\left( \begin{array}{cccc}f_0 & f_1 & \ldots & f_{k-1}\end{array}\right)
\left( \begin{array}{cccc}
b_1 & b_2 & \ldots & b_n \\
b_1a_1 & b_2a_2 & \ldots & b_na_n \\
\vdots & \vdots & \ddots & \vdots \\
b_1a_1^{k-1} & b_2a_2^{k-1} & \ldots & b_na_n^{k-1}
\end{array}\right) $$

\begin{rem}
Let $f\in L_k$ such that $f_{k-1}$ is the coefficient of $X^{k-1}$ in $f$. Then, the evaluation of the polynomial $f$ at $\infty$ is defined as follows: $f(\infty) = f_{k-1}$.
\end{rem}

\begin{defi}
Let $\mathbf a$ be an $n$-tuple of mutually distinct elements of $\mathbb P^1(\mathbb F_q)$
and $\mathbf b$ be an $n$-tuple of nonzero elements of $\mathbb F_{q}$.
Then, the \emph{generalized Reed-Solomon code} (in short, GRS code) of dimension $k$ defined by $\mathbf a$ and $\mathbf b$ is the image of the vector space $L_k$ under the evaluation $\mathrm{ev}_{\mathbf a, \mathbf b}$, that is
$$\mathrm{GRS}_k(\mathbf a, \mathbf b) = \left\{ \mathbf b * f(\mathbf a) \mid f\in L_k\right\} \hbox{ with } 0\leq k \leq n \leq q+1$$
The $n$-tuple $\mathbf a$ is called  the {\em evaluation-point sequence} and its elements $a_i$ are called \emph{code locators} and the values $b_i$ are called \emph{column multipliers}.
\end{defi}

\begin{rem}
From this notation of a GRS code it is not clear over which field the code is defined and from which field the elements $a_j$ and $b_j$
are chosen, they might be elements in an extension field $\mathbb F_{q^m}$.
We say that the code $\mathrm{GRS}_k(\mathbf a, \mathbf b)$ is {\em defined over} $\mathbb F_q$ if the evaluated polynomials have coefficients in
$\mathbb F_q$, $a_j \in \mathbb P^1(\mathbb F_q)$ and $b_j\in \mathbb F_q^*$ for all $j$.
So a limitation of GRS codes is the fact that the length $n$ is bounded by the size of the field $\mathbb F_q$ over which they are defined, indeed $n\leq q+1$. If necessary we take a sufficiently large extension of $\mathbb F_q$.
\end{rem}

GRS codes are MDS codes. Furthermore, the dual of a GRS code is also GRS, in particular $\mathrm{GRS}_k(\mathbf a, \mathbf b)^{\perp} = \mathrm{GRS}_{n-k}(\mathbf a, \mathbf b')$ for some $n$-tuple $\mathbf b'$ of nonzero elements of $\mathbb F_q$ that is explicitly known. These results are well-known and the interested reader is referred to \cite{huffman:2003, lint:1999, macwilliams:1977}.

Since the monomials $1, X, \ldots, X^{k-1}$ form a basis of $L_k$. Evaluating these monomials gives the canonical generator matrix of the code $\mathrm{GRS}_k(\mathbf a, \mathbf b)$ whose $i$-th row is the element
$\mathbf b* \mathbf a^i$ with $i=0, \ldots, k-1$ where we define by induction $\mathbf a^1=\mathbf a$ and $\mathbf a^{i+1} = \mathbf a * \mathbf a^i$.
By definition we may have $a_j=\infty$ for some $j \in \{ 1, \ldots, n\}$. Then, $\left(\begin{array}{cccc}0 & \ldots &0 & b_j\end{array}\right)^T$ is the $j$-th corresponding column vector of the canonical generator matrix.

\begin{rem}
\label{Rem-Main}
Consider MDS codes with parameters $[n,0]$, $[n,1]$, $[n,n-1]$ and $[n,n]$:
\begin{itemize}
\item The $\mathbb F_q$-linear codes with parameters $[n,0]$ and $[n,n]$ are the trivial codes $\{0\}$ and $\mathbb F_{q}^{n}$ which are not only MDS and dual to each other, but also GRS codes if $n\leq q+1$. Indeed, let $C=\{0\}$ then, $C = \mathrm{GRS}_0(\mathbf a, \mathbf b)$ for every $n$-tuple $\mathbf a$ of mutually distinct elements of $\mathbb P^1(\mathbb F_q)$ and every $n$-tuple $\mathbf b$ of nonzero elements of $\mathbb F_{q}$.
Thus, $C^{\perp}= \mathbb F_{q}^{n} = \mathrm{GRS}_{n}(\mathbf a, \mathbf b')$  for some $\mathbf b'\in \mathbb F_{q}^n$.

\item An MDS code $C$ with parameters $[n,1]$ defined over $\mathbb F_{q}$, is a GRS code if $n\leq q+1$. Indeed, $C$ is generated by a word $\mathbf b\in \mathbb F_{q}^{n}$ with nonzero entries. Thus $C = \mathrm{GRS}_1(\mathbf a, \mathbf b)$ where $\mathbf a$ is any $n$-tuple of mutually distinct elements of $\mathbb P^1(\mathbb F_{q})$.
Moreover its dual $C^{\perp}$, which is an MDS code with parameters $[n,n-1]$ is again a GRS code.
\end{itemize}
Therefore $\mathbb F_{q^m}$-linear MDS codes with parameters $[n,0]$, $[n,1]$, $[n,n-1]$ and $[n,n]$ are GRS codes defined over $\mathbb F_{q^m}$
if $m\geq \log_q(n-1)$.
\end{rem}

\begin{lemma}
\label{t=1}
Let $C$ be an $[n,n-2,3]$ code over $\mathbb F_q$. Then, $C$ is a GRS code that is defined over $\mathbb F_q$.
\end{lemma}

\begin{proof}Let $C$ be an $[n,n-2,3]$ code over $\mathbb F_q$.
Then, $C$ is an MDS code, so its dual is also MDS and is systematic at the first two positions.
Hence $C$ has a parity check matrix $H=\left(\begin{array}{c|c}I_2&P\end{array}\right)$, where $P$ is a $2 \times (n-2)$ matrix with entries $p_{ij}$ in $\mathbb F_q$.
Moreover all entries of $P$ and all determinants of $2 \times 2 $ minors of $P$ are nonzero, since $C$ is MDS.
Let $\mathbf x = \left(\begin{array}{ccccc}1 & 1 & p_{13} & \ldots & p_{1n}\end{array}\right)$.
by replacing $C$ by $\mathbf x*C$, that is $C^\perp$ by $\mathbf x^{-1}*C$,
we may assume that $p_{1j}=1$ for $3 \leq j \leq n$. Thus
$$
H_{C} = \left(\begin{array}{ccccc}
1 & 0 & 1 & \ldots & 1 \\
0 & 1 & p_{23} & \ldots & p_{2n}
\end{array}\right) \in \mathbb F_q^{2\times n}.
$$
Let $\mathbf a$ have entries $a_1 = 0$, $a_2 = \infty$ and $a_j = p_{2j}$ for $3\leq j \leq n$.
Then, $\mathbf a$ consists of mutually distinct entries in $\mathbb P^1(\mathbb F_q)$.
Let $\mathbf 1\in \mathbb F_q^n$ be the all-ones vector.
Then, $H$ is the canonical generator matrix of the code $\mathrm{GRS}_2(\mathbf a, \mathbf 1)$.
Therefore $C^{\perp}$ is a GRS code defined over $\mathbb F_q$. So this holds also for $C$.
\end{proof}

GRS codes can also be viewed as algebraic geometry (AG) codes on the projective line, that is the curve of genus zero. For this description we refer to \cite{stichtenoth:1993}.

\subsection{GRS codes and fractional transformations}

We can introduce polar coordinates on $\mathbb P^1(\mathbb F_q)$ via the map
$$\begin{array}{cccc}
\varphi : & \mathbb P^1(\mathbb F_q) & \longrightarrow & \mathbb F_q^2 \setminus \{(0,0)\}
\end{array}$$
where $\varphi(z) = (z:1)$ if $z \in \mathbb F_q$ and $\varphi(z) = (1:0)$ if $z=\infty$. The inverse of $\varphi$ is
$$\begin{array}{cccc}\psi: & \mathbb F_q^2 \setminus \{(0,0) \} &\longrightarrow & \mathbb F_q^* \times \mathbb P^1(\mathbb F_q)\end{array}$$
where $\psi(u,v) = (v, \frac{u}{v})$ if $v\neq 0$ and $\psi(u,v)=(u,\infty)$ if $v=0$. We will denote the components of $\psi$ by $\psi_1$ and $\psi_2$.

Let $\mathbb F_q[X,Y]_l$ denote the set of homogeneous polynomials over $\mathbb F_q$ of degree $l$ in two variables. Then, for any polynomial $f\in L_k$ the corresponding homogenous polynomial, denoted by $f_H$ and obtained by introducing an additional variable $Y$, belongs to $\mathbb F_q[X,Y]_{k-1}$, that is to say:

$$f_H(X,Y)=Y^{k-1}f\left(\frac{X}{Y}\right)\in \mathbb F_q[X,Y]_{k-1} \hbox { for all } f \in L_k \setminus L_{k-1}.$$
With this definition the evaluation of every polynomial on the projective line is given by
$$f(z) = f_H(\varphi(z)) \hbox{ for every } f \in L_k \hbox{ and }z \in \mathbb P^1(\mathbb F_q).$$

Let $\mathrm{GL}(n,\mathbb F_q)$ denote the linear group of $n\times n$ non-singular matrices with entries in
$\mathbb F_q$. Then, $\mathrm{GL}(2,\mathbb F_q)$ is a transformation group that act on $\mathbb F_q^2\setminus \{(0,0)\}$, on the projective line and also on the linear space $\mathbb F_q[X,Y]_l$ via the operators $\phi$ defined below for every $2\times 2$ non-singular matrix
$$M=\left(\begin{array}{cc} A & B \\ C & D\end{array}\right)\in \mathrm{GL}(2,\mathbb F_q).$$
\begin{enumerate}
\item $\phi\left(M,(u,v)\right) = (u,v) M^T = (Au+Bv, Cu+Dv)$, for all $(u,v) \in \mathbb F_q^2 \setminus \{(0,0)\}$.
\item $\phi\left(M,z\right) = \psi_2\left(\phi\left(M,\varphi(z)\right)\right)$, for all $z\in \mathbb P^1(\mathbb F_q)$.

In particular for $z \in \mathbb F_q$ we have that $\phi\left(M,z\right) = \frac{Az+B}{Cz+D}$ if $Cz+D \neq 0$, and $\infty$ otherwise. While for $z=\infty$ we obtain
$\phi\left(M,z\right) = \frac{A}{C}$ if $C \neq 0$, and $\infty$ otherwise.
The set of all these maps form the set of \emph{fractional transformation of the projective line}.

\item $\phi\left(M^{-1},P(X,Y)\right) = P(AX+BY, CX+DY)$, for all $P\in \mathbb F_q[X,Y]_l$.

    Therefore we can extend the previous operator to any polynomial $f\in L_k$ by setting:
    $$\phi\left(M,f(z)\right) = \phi(M, f_H(\varphi(z))).$$
\end{enumerate}

Now we define the map
$$\begin{array}{cccc}
\theta: & \mathrm{GL}(2,\mathbb F_q) \times \mathbb P^1(\mathbb F_q) & \longrightarrow & \mathbb F_q^{*}
\end{array}$$
by $\theta(M,z) = \psi_1(\phi(M,\varphi(z)))$. Therefore for $z \in \mathbb F_q$ we have that $\theta(M,z) = Cz+D$ if $Cz+D \neq 0$, and $Az+D$ otherwise. Whereas for $z=\infty$ we have $\theta(M,z) = C$ if $C\neq 0$, and $A$ otherwise.

The following proposition shows that different values of $\mathbf a$ and $\mathbf b$ give rise to the same GRS code
and that the pair $(\mathbf a, \mathbf b)$ is unique up to the action of fractional transformations.

\begin{prop}
\label{GRS:Automorphism}
Let $2 \leq k \leq n-2$. Then, $\mathrm{GRS}_k(\mathbf a, \mathbf b) = \mathrm{GRS}_k(\mathbf c, \mathbf d)$ if and only if there exists $M\in \mathrm{GL}(2, \mathbb F_q)$ and $\lambda \in \mathbb F_q^*$ such that
$$
\begin{array}{ccc}
c_i = \phi(M,a_i) & \hbox{ and } & d_i = \lambda\theta(M, a_i)^{k-1} b_i
\end{array}
$$
for all $i=1, \ldots ,n$.
\end{prop}

\begin{proof}
See \cite[Theorem 4]{dur:1987} or \cite[Theorem 2.6]{huffman:1998}.
\end{proof}

\begin{rem}\label{r-3points}Let $M\in \mathrm{GL}(2,\mathbb F_q)$ and $\lambda \in \mathbb F_q^*$. Then,
$M$ and $\lambda M$ define the same fractional transformation of the projective line.

Let $\left(a_1,a_2,a_3\right)$ be a triple of mutual distinct points on the projective line.
Then, there is a unique fractional transformation of the projective line that sends
$(a_1,a_2,a_3)$ to $(0,1,\infty)$.
Hence there is a unique fractional transformation of the projective line that sends
$(a_1,a_2,a_3)$ to another  triple $(c_1,c_2,c_3)$ of three distinct points on the projective line.
\end{rem}

\begin{cor}
\label{GRS:Automorphism-cor}
Let $2 \leq k \leq n-2$ and $\mathrm{GRS}_k(\mathbf a, \mathbf b) = \mathrm{GRS}_k(\mathbf c, \mathbf d)$.
If $\mathbf a$ and $\mathbf c$  coincide at $3$ distinct positions,
then $\mathbf a = \mathbf c$ and $\mathbf d = \lambda \mathbf b$ for some $\lambda \in \mathbb F_q^*$.
\end{cor}

\begin{proof}
Suppose $2 \leq k \leq n-2$ and $\mathrm{GRS}_k(\mathbf a, \mathbf b) = \mathrm{GRS}_k(\mathbf c, \mathbf d)$
and that $\mathbf a$ and $\mathbf c$  coincide at $3$ distinct positions.
Then, there exists $M\in \mathrm{GL}(2, \mathbb F_q)$ and $\lambda \in \mathbb F_q^*$ such that
$c_i = \phi(M,a_i)$ and $d_i = \lambda\theta(M, a_i)^{k-1} b_i$ for all $i$ by Proposition \ref{GRS:Automorphism}.
Now $c_i=a_i$, $c_j=a_j$, and $c_l=a_l$ for some $1 \leq i<j<l \leq n$.
The identity map and $\phi$ are fractional transformations that leave $a_i$, $a_j$ and $a_l$ fixed.
Hence $\phi $ is the identity map by Remark \ref{r-3points}. So $\mathbf a = \mathbf c$ and $\mathbf d = \lambda \mathbf b$.
\end{proof}

\subsection{GRS and generalized Cauchy codes}
In this section we refer to  \cite{oberst:1985,dur:1987}.
Let $\mathbf a$ be an $n$-tuple of mutually distinct elements of $\mathbb P^1(\mathbb F_q)$, so $n\leq q+1$, and let $\mathbf c$ be an $n$-tuple of nonzero elements of $\fq$.
Define
$$\begin{array}{ccccc}
[a_i, a_j] = a_i - a_j, &
[\infty, a_j] = 1 &
\hbox{ and } &
[a_i, \infty] = -1 &
\hbox{ for } a_i, a_j \in \mathbb F_q
\end{array}$$

\begin{defi}
Let $A(\ba,\mathbf c)$ be the $k\times (n-k)$ matrix with the entries
$$
\frac{c_{j+k}}{c_i[a_{j+k},a_i]} \ \mbox{ for } 1\leq i \leq k,\ 1\leq j \leq n-k.
$$
The {\em generalized Cauchy} code  $C_k(\ba, \mathbf c)$ is an $[n,k,n-k+1]$ code defined by the generator matrix  $\left(\begin{array}{c|c}I_k&A(\ba,\mathbf c)\end{array}\right)$.
\end{defi}

\begin{prop}\label{pcauchy}
Let $\ba$ be an $n$-tuple of mutually distinct elements of $\mathbb P^1(\mathbb F_q)$, and
let $\bb $ be an $n$-tuple of nonzero elements of $\fq $.
Let
$$
c_i= \left\{
\begin{array}{ll}
b_i\prod_{t=1,t\not=i}^k[a_i,a_t] & \mbox{ if } 1 \leq i \leq k,\\
b_i\prod_{t=1}^k[a_i,a_t]         & \mbox{ if } k+1\leq i \leq n.
\end{array}
\right.
$$
Then, $\mathrm{GRS}_k(\ba ,\bb) = C_k(\ba ,\bc )$.
\end{prop}

\begin{proof} See \cite{oberst:1985} and \cite[Theorem 2]{dur:1987}.
\end{proof}

\begin{prop}
\label{pGRSdefoverFq}
Let $2 \leq k \leq n-2$.
If $C=\mathrm{GRS}_k(\ba, \bb)$ is defined over $\fqm$ and has a generator matrix over $\fq $,
then $C$ is defined over $\fq$.
\end{prop}

\begin{proof}
The following proof is similar to the proof of Sidelnikov-Shestakov \cite{sidelnikov:1992}
in their attack of the code based cryptosystem using GRS codes,
where they find explicitly and efficiently the pair $(\ba ,\bb)$ from a generator matrix of a $\mathrm{GRS}_k(\ba ,\bb)$ code.

Let $C$ be a GRS code defined over $\fqm$ with generator matrix $G$ defined over $\mathbb F_q$.
Hence $C=\mathrm{GRS}_k(\mathbf a',\mathbf b')$ for $n$-tuples $\mathbf a'$ and $\mathbf b'$ defined over $\fqm$.
We will show that there are $n$-tuples $\mathbf a$ and $\mathbf b$ defined over $\mathbb F_q$ such that
$C=\mathrm{GRS}_k(\mathbf a, \mathbf b)$.
By Proposition \ref{pcauchy} a GRS code is defined over $\fqm$ if and only if its corresponding Cauchy code $C_k(\mathbf a, \mathbf c)$ is defined over $\fqm$.

Without loss of generality we may assume that:
\begin{itemize}
\item $c_1 = 1$, since $c_1 \in \fqm^*$ and $C_k(\mathbf a, \mathbf c') = C_k(\mathbf a, \mathbf c)$ with $c'_j = \frac{c_j}{c_1}$ for all $j$.
\item $a_1=0$, $a_2 = 1$ and $a_{k+1}=\infty$ after applying fractional transformation, see Proposition \ref{GRS:Automorphism}.
\end{itemize}

The matrix $G$ has entries in $\mathbb F_q$. So its reduced echelon form which is of the form $\left( \begin{array}{c|c} I_k & P\end{array}\right)$ has also entries in $\mathbb F_q$. Thus the entries $p_{ij}$ of $P$ are in $\mathbb F_q$.
By Proposition \ref{pcauchy} we have that:
$$p_{ij} = \frac{c_{k+j}}{c_i [a_{k+j}, a_i]} \hbox{ for } 1\leq i \leq k \hbox{ and } 1\leq j \leq n-k$$

Thus:
\begin{enumerate}
\item $p_{i1} = \frac{c_{k+1}}{c_i}\in \fq$ for all $1\leq i \leq k$, since $a_{k+1}= \infty$.

So $c_{k+1}\in \fq$, since $c_1=1$. And therefore, $c_{i}\in \fq$ for all $1\leq i \leq k+1$.
\item\label{it:2} $p_{1j} = \frac{c_{k+j}}{a_{k+j}}\in \fq$ for all $1\leq j \leq n-k$, since $c_1 = 1$ and $a_1 = 0$.
\item\label{it:3} $p_{2j} = \frac{c_{k+j}}{c_2 ( a_{k+j}-1)}\in \fq$ for all $1\leq j \leq n-k$, since $a_2 = 1$.

So $\frac{c_{k+j}}{( a_{k+j}-1)}\in \fq$ for all $1\leq j \leq n-k$, since $c_2 \in \fq $.

\item[] The quotient $\frac{p_{1j} }{p_{2j}}$ of the elements in (\ref{it:2}) and (\ref{it:3}) gives that
$1-a_{k+j}^{-1} \in \fq$ for all $1\leq j \leq n-k$. So $a_{k+j}\in \fq $ and therefore, by (\ref{it:2}), also $c_{k+j} \in \fq $ for all $1\leq j \leq n-k$.
\item Finally $p_{i,k+2}=\frac{c_{k+2}}{c_i(a_{k+2}-a_i)}\in \fq $ for all $1\leq i \leq k$.

Therefore $a_i \in \fq $, since $c_{k+2}, a_{k+2} \in \fq $ and $c_i \in \fq$  for all $1\leq i \leq k$.
\end{enumerate}

We have thus proved that $C$ is a Cauchy code defined over $\mathbb F_q$.
Hence $C$ is also a GRS code defined over $\mathbb F_q$.
\end{proof}

\begin{rem}Let $C$ be an $\fq$-linear code. The extended code $C\otimes \fqm $ is the  $\fqm$-linear code generated by $C$. If $G$ is a generator matrix of $C$ with entries in $\fq$. Then, $G$ is also a generator matrix of $C\otimes \fqm $
but now over $\fqm$.
In the following we will encounter the situation that $C$ is an $\fq$-linear code of length $n$
where we have to extend the field $\fq $ to $\fqm$,
for instance to ensure that $C\otimes \fqm $ has a codeword of weight $n$,
and arrive to the conclusion that $C\otimes \fqm $ is a GRS code defined over $\fqm$.
Then, $C$ is also a GRS code defined over $\fq $ by Proposition \ref{pGRSdefoverFq}.
\end{rem}

\subsection{GRS codes and the rational normal curve}

Let us mention the connection between linear codes and affine or projective varieties by menas of projective systems \cite[\S 1.1.2]{tsfasman:1991}.
An $n$-tuple of points $(P_1, \ldots, P_n)$ in ${\Bbb P}^r(\mathbb F_q)$  is a {\em projective system} if not all these points lie in a hyperplane.
A projective system is called an $n$-{\em arc} in ${\Bbb P}^r(\mathbb F_q)$ if the points of the system are in {\em general position},
that is if no $r+1$ points of them lie on a hyperplane.

Let $\mathcal P=(P_1, \ldots, P_n)$ be a projective system in ${\Bbb P}^r (\mathbb F_q)$ where $P_j$ is given by the homogeneous coordinates $(p_{0j}:p_{1j}:\ldots :p_{rj})$. We define the associated $(r+1)\times n$-matrix $G_\mathcal P$ as the matrix with $P_j^T$ as $j$-th column. Then, $G_\mathcal P$ has rank $r+1$, since not all points lie in a hyperplane. That is, $G_\mathcal P$ is the generator matrix of a nondegenerate $[n,r+1]$ code over $\mathbb F_q$.

Conversely, let $C$ be a nondegenerate $[n,k]$ code over $\mathbb F_q$ with generator matrix $G$. Take the columns of $G$ as homogeneous coordinates of points in ${\Bbb P}^{k-1}(\mathbb F_q)$, this gives the associated projective system $\mathcal P_G$ over $\mathbb F_q$ of $G$. Furthermore the code has minimum distance $d$ if an only if $n-d$ is the maximal number of points of $\mathcal P_G$ in a hyperplane of ${\Bbb P}^{k-1}(\mathbb F_q)$.
Hence $n$-arcs in ${\Bbb P}^r(\mathbb F_q)$ correspond one-to-one to generalized equivalence classes of MDS codes of length $n$ and dimension $r+1$ over
$\mathbb F_q$. See \cite[Remark 13.9]{hirschfeld:2008}.

Consider the map ${\Bbb P}^1 \rightarrow {\Bbb P}^r$ given by
$$
\begin{array}{ccc}
(x:y) & \longmapsto & (x^r : x^{r-1}y: \ldots : x^{r-i}y^{i}: \ldots : xy^{r-1}: y^r).
\end{array}
$$
Then, the image of this map in ${\Bbb P}^r$ is a curve of degree $r$ and is denoted by $\mathcal X_r$.
A {\em rational normal curve} in ${\Bbb P}^r$ of degree $r$ is a projective transformation of $\mathcal X_r$. See \cite[Example 7.35]{hirschfeld:2008}

The columns of the canonical matrix of $GRS_{r+1}(\mathbf a, \mathbf b)$, considered as homogeneous coordinates, are points of the curve
$\mathcal X_r(\mathbb F_q)$.
Therefore a GRS code of dimension $r+1$ can be described as a projective system of points on a rational normal curve of degree $r$ in ${\Bbb P}^r(\mathbb F_q)$.
The converse is also true.

Now we state a classical result on rational normal curves.

\begin{prop}\label{prop-RNC}
Through any $r+3$ points in ${\Bbb P}^r(\mathbb F_q)$ in general position there passes a unique rational normal curve.
\end{prop}

\begin{proof}
See \cite{bordiga:1885,carlini:2007,castelnuovo:1885,piggott:1947,veronese:1882} and \cite[p. 530]{griffiths:1978}.
\end{proof}

\section{Error-correcting pairs}
\label{ECP:section}

We present in this section a general decoding method for linear codes that have an \emph{``Error-Correcting Pair (ECP)''}.
The notion of an error correcting pair for a linear code was introduced independently by Pellikaan \cite{pellikaan:1988,pellikaan:1992} and K\"otter in \cite{koetter:1992}. These articles show that there exists a decoding algorithm that corrects up to $t$ errors with complexity $\mathcal O(n^3)$ for every linear code defined over $\mathbb F_q^n$ that has a $t$-error correcting pair.
Indeed, ECP's provide a unifying point of view for several classical bounded distance decoding for algebraic codes.

\begin{defi}
Let $C$ be an $\mathbb F_q$-linear code of length $n$. The pair $(A,B)$ of $\mathbb F_{q^m}$-linear codes of length $n$ is called a \emph{$t$-error correcting pair} (ECP) for $C$ if the following properties holds:
\begin{multicols}{2}
\begin{enumerate}[E.$1$]
 \item\label{it:E1} $(A*B) \perp C$,
 \item\label{it:E2} $k(A) > t$,
 \item\label{it:E3} $d(B^{\perp}) > t$,
 \item\label{it:E4} $d(A) + d(C) > n$.
 \end{enumerate}
\end{multicols}
\end{defi}

Broadly speaking: given a positive integer $t$, a $t$-ECP for a linear code $C\subseteq F_q^n$ is a pair of linear codes $(A,B)$ satisfying that $A*B \subseteq C^{\perp}$ together with several inequalities relating $t$ and the dimensions and (dual) minimum distances of $A$, $B$ and $C$.
Furthermore note that if the fourth property (E.$\ref{it:E4}$) is replaced by the statements presented below then, again $(A,B)$ is a $t$-ECP for $C$ and the minimum distance of such linear code is at least $2t+1$.
\begin{enumerate}[E.$5$]
 \item\label{it:E5} $d(A^{\perp}) > 1$ or equivalently $A$ is a non-degenerated code,
 \item\label{it:E6} $d(A) + 2t > n$.
\end{enumerate}

\subsection{Existence of ECP for several families of codes}

The existence of ECP for generalized Reed-Solomon and Algebraic codes  was shown in \cite{pellikaan:1996} and for many cyclic codes Duursma and K\"otter in \cite{duursma:1994}
have found ECP which correct beyond the designed BCH capacity.

\begin{ex}
\label{ECP:exampleGRS}
The class of GRS codes belongs to the class of linear codes that have a $t$-ECP. Indeed let $\mathbf a$ be an $n$-tuple of mutually distinct elements of $\mathbb P^1(\mathbb F_q)$ and $\mathbf b$ be an $n$-tuple of nonzero elements of $\mathbb F_q$.
Let $C = \mathrm{GRS}_k(\mathbf a, \mathbf b)$ and let $\mathbf b'$ be such that $C^\perp = \mathrm{GRS}_k(\mathbf a, \mathbf b')$.
Then, the codes
$$\begin{array}{ccc}
A=\mathrm{GRS}_{t+1}(\mathbf a, \mathbf b') & \hbox{ and } &
B=\mathrm{GRS}_t(\mathbf a, \mathbf 1),
\end{array}$$
form a $t$-ECP for the code $C$ with $t=\left\lfloor\frac{n-k}{2} \right\rfloor$.

Conversely, let $\mathbf a$ be an $n$-tuple of mutually distinct elements of $\mathbb P^1(\mathbb F_q)$ and $\mathbf b$ and $\mathbf c$ be $n$-tuples of nonzero elements of $\mathbb F_q$. Define
$$\begin{array}{ccc}
A=\mathrm{GRS}_{t+1}(\mathbf a, \mathbf b) & \hbox{ and }&
B = \mathrm{GRS}_{t}(\mathbf a, \mathbf c)\end{array}$$
Then, $(A,B)$ is a $t$-ECP for $C=\mathrm{GRS}_{2t}(\mathbf a, \mathbf b*\mathbf c)^{\perp}$.
\end{ex}

\begin{ex}
If $D$ is a code that has a $t$-ECP and $C$ is a subcode of the code $D$. Then, $C$ also has a $t$-ECP.
In particular, let $C$ be a subcode of a GRS code, then by Example \ref{ECP:exampleGRS} the GRS code has a $t$-ECP which is also a $t$-ECP for $C$.

Moreover, a Goppa code associated to a Goppa polynomial of degree $r$ can be viewed as an alternant code, that is a subfield subcode of a GRS code of codimension $r$. Thus Goppa codes have also an $\left\lfloor\frac{r}{2}\right\rfloor$-error correcting pair. In the binary case with an associated square free polynomial the Goppa code has an $r$-ECP.
\end{ex}

\begin{ex}
An Algebraic-Geometry code defined on a curve of genus $g$ with designed minimum distance $d^*$ has a $t$-ECP over $\mathbb F_q$ with $t=\left\lfloor \frac{d^*-1-g}{2}\right\rfloor$ by \cite[Theorem 3.3]{pellikaan:1992} and \cite[Theorem 1]{pellikaan:1989}. Furthermore if $e$ is sufficiently large, then there exists a $t$-ECP over $\mathbb F_{q^e}$ with $t=\left\lfloor \frac{d^*-1}{2}\right\rfloor$ by \cite[Proposition 4.2]{pellikaan:1996}.
\end{ex}

\begin{ex}
By \cite[Corollary 3.4]{pellikaan:1996} any pair of MDS codes $(A,B)$ with parameters $[n,t+1,n-t]$ and $[n,t,n-t+1]$, respectively, form a $t$-ECP for $C=(A*B)^{\perp}$ where the dimension of $C$ is almost always equal to $n-t(t+1)$ by \cite[Appendix A]{marquez:2012c}. Moreover, since random codes are MDS when $q$ is sufficiently larger than $n$, taking random codes $A$ and $B$ of length $n$ and dimension $t+1$ and $t$, respectively, gives a $t$-ECP for $C$ defined again as the dual of the product code $A*B$.
\end{ex}

\section{Puncturing, shortening and gluing}
\label{Puncturing:Section}

Before going through the proof of the main theorem we will first have a brief review of classical methods of constructing a shorter code out of a given one. For a fuller treatment we refer the reader to \cite{huffman:2003}.

Let $\mathbf x \in \mathbb F_q^n$ and $J$ be a subset of $\{1, \ldots, n\}$ consisting of $m$ integers $i_1, \ldots, i_m$ with $1\leq i_1< \cdots < i_m \leq n$, we denote by
$$
\mathbf x_{J} = (x_{i_1}, \ldots, x_{i_m}) = \left( x_{j}\mid j\in J\right) \in \mathbb F_q^m
$$
the restriction of $\mathbf x$ to the coordinates indexed by $J$. Let $\overline{J}$ be the relative complement of $J$ in $\{1, \ldots, n\}$.

In what follows $C$ stands for an $[n,k,d]$ code.

The process of deleting columns from a generator matrix of a given linear code $C$, or equivalently, deleting the same set of coordinates in each codeword of $C$ is called \emph{puncturing}. Consider a subset $J\subseteq \{1, \ldots, n\}$ with $m$ elements and suppose $G$ is a generator matrix of $C$, the \emph{punctured code} $C_{J}$ is the linear code generated by the rows of the $k\times (n-m)$ submatrix of $G$ composed of the $\overline{J}$-indexed columns of $G$, in other words $C_J$ is the set of codewords of $C$ restricted to the positions of $\overline{J}$, that is
$C_{J} =\{\mathbf c_{\overline{J}} \mid \mathbf c \in C\}$.

Let $m<d$. If $\bc , \bc' \in C$ and $\bc_{\overline{J}}=\bc'_{\overline{J}}$, then $\bc =\bc'$.
Therefore, if $A$ and $A'$ are subcodes of $C$ such that $A_J=A'_J$, then $A=A'$.

The punctured code $C_{J}$ is an $[n-m, k(C_{J}), d(C_{J})]$ code with
$$\begin{array}{ccc}
d - m \leq d(C_{J}) \leq d & \hbox{ and } &
k - m \leq k(C_{J}) \leq k
\end{array}.$$
Moreover if $m<d$ then $k(C_J) = k$.

Whereas the process of deleting columns from a parity check matrix of $C$ is known as \emph{shortening}. The shortened code $C^J$ is obtained by puncturing at $J$ the set of codewords that have a zero in the $J$-locations, that is
$$
C^{J} = \{\mathbf c_{\overline{J}} \mid \mathbf c \in C \hbox{ and } \mathbf c_{J} = 0\}.
$$

To compute a generator matrix of $C^{\{j\}}$, we first find a generator matrix of the original code that has a unique row in which the $j$-th column is nonzero. Then, we delete that row and the $j$-th column leaving the sought matrix. This process can be generalized to subsets with more than one element. If $J$ consist of $m$ elements then the shortened code $C^{J}$ is an $[n-m, k(C^{J}), d(C^{J})]$ linear code with
$$
\begin{array}{ccc}
k-m \leq k(C^{J}) \leq k & \hbox{ and } &
d \leq d(C^{J})
\end{array}
$$

Let $C^{\{n\}}$ be obtained from $C=GRS_k(\ba,\bb)$ by shortening on the last coordinate.
Then, $C^{\{n\}}=GRS_{k-1}(\ba',\bb'*\bv')$, where $\ba'= (a_1, \ldots, a_{n-1})$,
$\bb'= (b_1, \ldots, b_{n-1})$ and $\bv'$ has entries $a_j-a_n$ for $j=1, \ldots ,n-1$.

Take notice of some properties of these operations such that the shortened code $C^{J}$ is a subcode of the punctured code $C_{J}$, that is $C^{J}\subseteq C_{J}$,
and the fact that shortening a code is dual to puncturing it, to be precise
$$
\begin{array}{ccc}
\left(C_{J}\right)^{\perp} = \left( C^{\perp}\right)^{J} & \hbox{ and } &
\left(C^{J} \right)^{\perp} = \left( C^{\perp}\right)_{J}.
\end{array}
$$

\begin{lemma}
\label{MT:lemma}
Let $C$ be an $[n,k]$ MDS code and $J$ be a subset of $\{1, \ldots, n\}$ with $m$ elements such that $n - m \geq k$. Then, $C_J$ and $C^J$ are MDS codes with parameters
$\begin{array}{ccc}
[n - m, k] & \hbox{ and }&
[n - m, k - m]
\end{array}$,
respectively.
\end{lemma}

\begin{proof}
Let $C$ be an MDS code. Then, every $k$-tuple of columns of a generator matrix of $C$ are linearly independent, thus $k = k(C_J)$ unless $n - m < k$ which contradicts the hypothesis made. Moreover we have that
$$d - m \leq d(C_J) \leq d.$$
Hence $d(C_J) \leq n(C_J) - k(C_J) + 1 = d - m$, using the Singleton bound. Hence, $C_J$ is an MDS code with the desired parameters.

On the other hand, using the above known bound for the parameters of the shortened code $C^J$ together with the Singleton bound we have that
$$
\begin{array}{c}
k - m \leq k(C^J) \leq k ~~\hbox{ and } \\
n -k + 1 = d \leq d(C^J) \leq n - m + k(C^J) + 1.
\end{array}
$$
There is only one possibility that satisfies the two inequalities and that coincides with the statement of our lemma.
\end{proof}

\begin{prop}[Gluing Property]
\label{puncture}
Let $C$ be an $[n,k]$ MDS code. Let $I$ and $J$ be two disjoint subsets of $\{ 1, \ldots ,n \}$.
Suppose $2 \leq k \leq n-|I\cup J|-2$.
If $C_{I}$ and $C_{J}$ are GRS codes, then $C$ is a GRS code.
\end{prop}

\begin{proof}
After a permutation of $\{ 1, \ldots ,n \}$ we may assume without loss of generality  that
$I= \{ 1, \ldots ,i \}$ and $J= \{ n-j+1, \ldots ,n \}$ with $i+j\leq n$, since $I$ and $J$ are disjoint.

By Lemma \ref{MT:lemma}, the code $C_I$ is a GRS code with parameters $[n-i, k]$,
so there exists a vector $\mathbf u=(u_{i+1}, \ldots, u_n)$ of $(n-i)$ mutually distinct elements in $\mathbb F_q$ and a vector $\mathbf v = (v_{i+1}, \ldots, v_n)$ of $(n-i)$ nonzero elements in $\mathbb F_q$ such that $C_I = \mathrm{GRS}_{k}(\mathbf u, \mathbf v)$.

Similarly, by Lemma \ref{MT:lemma}, the code $C_J$ is a GRS code with parameters $[n-j, k]$, so there exists a vector $\mathbf r= (r_1, \ldots, r_{n-j})$ of $(n-j)$ mutually distinct elements in $\mathbb F_q$ and a vector $\mathbf s = (s_1, \ldots, s_{n-j})$ of $(n-j)$ nonzero elements in $\mathbb F_q$ such that
$C_J = \mathrm{GRS}_{k}(\mathbf r, \mathbf s)$.

We next consider the vectors $\mathbf u '=(u_{i+1}, \ldots, u_{n-j})$ and $\mathbf v '= (v_{i+1}, \ldots, v_{n-j})$ of $(n-j-i)$ mutually distinct elements in $\mathbb F_q$; and the vectors $\mathbf r '=(r_{i+1}, \ldots, r_{n-j})$ and
$\mathbf s '= (s_{i+1}, \ldots, s_{n-j})$ of $(n-j-i)$ nonzero elements in $\mathbb F_q$.

Let $K =I \cup J$. Then,
\begin{itemize}
\item $C_K=\left( C_I\right)_J$ is obtained from $C_I$ by puncturing on the $j$ right-most coordinates thus by Lemma \ref{MT:lemma}, $C_K = \mathrm{GRS}_{k}(\mathbf u ', \mathbf v ')$.
\item $C_K=\left( C_J\right)_I$ is also obtained from $C_J$ by puncturing on the $i$ left-most coordinates, that is, by Lemma \ref{MT:lemma}, $C_K = \mathrm{GRS}_{k}(\mathbf r ', \mathbf s')$.
\end{itemize}

Furthermore we have that $2\leq k(C_K) \leq n(C_K)-2$, since
$$\begin{array}{c}
k(C_K) = k, ~~ n(C_K) - 2 = n-|K|-2 \\
\hbox{ and, by assumption, } 2\leq k\leq n-|K|-2.
\end{array}$$
Moreover, by Proposition \ref{GRS:Automorphism}, a representation of $C_K$ as a GRS code is unique up to a fractional map of the projective line that induces an automorphism of the code. Hence we may assume that $\mathbf u ' = \mathbf r '$ and $\mathbf v ' = \mathbf s '$.

Let $\mathbf a =(r_1, \ldots ,r_{n-j},u_{n-j+1}, \ldots ,u_n)$ and $\mathbf b =(s_1, \ldots ,s_{n-j},v_{n-j+1}, \ldots ,v_n)$, or equivalently,
$\mathbf a = (r_1, \ldots ,r_i,u_{i+1}, \ldots ,u_n)$ and $\mathbf b =(s_1, \ldots ,s_i,v_{i+1}, \ldots ,v_n)$.
The points with homogeneous coordinates $(1:a_j:\cdots :a_j^{k-1})$ in the projective space correspond to the projective system associated to the generator matrix $G_{\ba}$ of the code $C$ which is MDS. Hence $\mathbf a$ consists of $n$ mutually distinct elements in $\mathbb P^1(\mathbb F_q)$, $\mathbf b$ is a vector of $n$ nonzero elements in $\mathbb F_q$ and $C=\mathrm{GRS}_{k}(\mathbf a ,\mathbf b)$ is a GRS code.
\end{proof}

\begin{rem}
In the following lines we give a second proof of the above result by means of rational normal curves. 
The associated projective system of the MDS code $C$ is
an $n$-tuple $\mathcal P_G = (P_1,\ldots ,P_n)$ in general position.
\begin{itemize}
\item The code $C_I$ is a GRS code so it corresponds to the projective system $\mathcal P_I = (P_{i+1},\ldots ,P_n)$ that lies on a rational normal curve $\scrx_I$ in $\mathbb P^r$ of degree $r=k-1$.
\item Similarly, the code $C_J$ is a GRS code, so it corresponds to the projective system $\mathcal P_J = (P_1,\ldots ,P_{n-j})$ that lies on a rational normal curve $\scrx_J$ in $\mathbb P^r$ of degree $r$.
\item Let $K=I\cup J$. The code $C_{K}$ corresponds to the projective system $\mathcal P_K=(P_{i+1},\ldots ,P_{n-j})$ that lies both on $\scrx_I$ and $\scrx_J$. 
\end{itemize}
But $k\leq n-|I\cup J|-2= n-(i+j)-2$. So $n-(i+j)\geq r+3$.
Hence, by Proposition \ref{prop-RNC}, $\mathcal P_K = (P_{i+1},\ldots ,P_{n-j})$ determines a unique rational normal curve: $\scrx = \scrx_I=\scrx_J$.
Or equivalently, all the points $(P_1,\ldots ,P_n)$ are on the rational normal curve $\scrx$. Therefore, $C$ is a GRS code.
\end{rem}

The assumption $2 \leq k \leq n-|I\cup J|-2$ in the previous proposition is essential as the following examples will show.

\begin{ex}
Let $\mathcal Q_1$ and $\mathcal Q_2$ be two distinct irreducible conics in the projective plane over $\mathbb F_q$, that is
rational normal curves of degree $2$ in $\mathbb P^2(\mathbb F_q)$.
By Bezout's Theorem, these conics intersect in at most $4$.
Choose the conics in such a way that they intersect in  exactly $4$ points, denoted by $P_3, P_4, P_5 ,P_6$.

Let $P_1, P_2$ two points in $\mathcal Q_1(\mathbb F_q)$ and let
$P_7, P_8$ two points in $\mathcal Q_2(\mathbb F_q)$ such that no three of the points $P_1,P_2,\ldots ,P_8$ are collinear -
this is possible if $q$ is sufficiently large.
Hence, $(P_1,P_2,\ldots ,P_8)$ is an $8$-arc in ${\Bbb P}^2(\mathbb F_q)$. We define $C$ as the corresponding $[8,3,6]$ MDS code.

Let $I=\{7 ,8\} $ and $J=\{1,2 \} $. Then, $C_{I}$ and $C_{J}$
are both GRS codes, since they are $6$-arcs that lie on the conics,
that is rational normal curves of degree $2$, $\mathcal Q_1$ and $\mathcal Q_2$ respectively.
But $C$ is not a GRS code, since $5$ points in general position determine a unique conic that passes through these points.
Note that the condition $2\leq k\leq n- |I\cup J|-2$ is not verified.
\end{ex}

\begin{ex}
Let $q\geq 8$ be a power of two. Let $\ba$ be a $5$-tuple of mutually distinct elements of $\fq $.
Let $C$ be the code with generator matrix
$$
G=
\left(
\begin{array}{ccccccc}
 1   & 1   & 1   & 1   & 1   &0&0 \\
a_1  &a_2  &a_3  &a_4  &a_5  &0&1 \\
a_1^2&a_2^2&a_3^2&a_4^2&a_5^2&1&0 \\
\end{array}
\right)
$$
Consider the columns of $G$ as points in the projective plane with homogeneous coordinates $x_0$, $x_1$ and $x_2$
corresponding to the first, second and third row, respectively.
Then, the $6$ points corresponding to the first $6$ columns lie on a unique
conic with equation $x_0x_2=x_1^2$. The last column is the nucleus of this conic, that is the
unique point through which all the tangent lines of the conic pass.
Hence $C$ is an MDS code with parameters $[7,3,5]$. But $C$ is not a GRS code, since the nucleus does not lie on the conic.

Let $I=\{ 6,7\}$ and $J=\{1,2\}$. Then, $C_I$ and $C_J$ are both GRS codes since through any $5$ points in $\mathbb P^2(\mathbb F_q)$ in general position there passes a unique rational normal curve (a conic), see Proposition \ref{prop-RNC}.
So we have a second example where $C_I$ and $C_J$ are both GRS codes but they do not \emph{``glue together''} to form a GRS code.
\end{ex}

\begin{ex}
This example comes from a complete $9$-arc in projective space of dimension 4 over $\ff_9$  \cite{glynn:1986},
or equivalently a $[9,4,6]$ code over $\ff_9$. Let $C$ be the dual code with parameters  $[9,5,5]$ given by the parity check matrix
\cite[p. 49]{glynn:1986}
$$
H=
\left(
\begin{array}{ccccccccc}
1 & 0 & 0 & 0 & 1 & 1        & 1        & 1        & 1  \\
0 & 1 & 0 & 0 & 1 & \sigma^5 & \sigma^4 & \sigma   & \sigma^2   \\
0 & 0 & 1 & 0 & 1 & \sigma^4 & \sigma   & \sigma^3 & \sigma^5   \\
0 & 0 & 0 & 1 & 1 & \sigma   & \sigma^7 & \sigma^6 & \sigma^4   \\
\end{array}
\right),
$$
where $\ff_9$ is obtained from $\ff_3$ by adjoining $\sigma $ such that $\sigma^2=\sigma +1$.
Now $C$ is not a GRS code, since its corresponding $9$-arc is complete, that is it cannot be extended to a $10$-arc, whereas
any GRS code over $\ff_9$ can be extended to an GRS code of length $10$.
The code $C$ has no $2$-ECP over an extension of $\ff_9$ by \cite[Theorem 6.5]{pellikaan:1996} and Proposition \ref{t=2}.
\end{ex}

\section{Main Theorem}
\label{MT:section}
It was proved in \cite[Corollary5.2]{pellikaan:1996} that every linear code over $\mathbb F_q$ is contained in an MDS code of the same minimum distance over some finite extension of $\mathbb F_q$. Therefore describing  linear codes that have an error correcting-pairs is equivalent to characterizing  those MDS codes that have an ECP.
The result presented in \cite{pellikaan:1996} shows that if  $t=1$ then every MDS code has a $t$ error correcting pair. Furthermore an $[n,n-4,5]$ code $C$ has a $2$-ECP if and only if $C$ is a GRS code \cite[Theorem 6.5]{pellikaan:1996}.

\begin{prop}
\label{t=2}
Let $C$ be an $[n,n-4,5]$ code over $\mathbb F_q$ that has $(A,B)$ as a $2$-ECP  over a finite extension of $\mathbb F_q$.
Then, $C$ is a GRS code defined over $\fq$. Moreover the codes $A,B$ are GRS codes with the same evaluation-point sequence as $C$ and
$B=(A*C)^\perp$.
\end{prop}

\begin{proof}That $C$ is a GRS code was proved in \cite[Theorem 6.5]{pellikaan:1996} as mentioned before.
For completeness sake, we include another proof together with the fact that $A$ and $B$ are GRS codes with the same evaluation-point sequence.

Let $(A,B)$ be a $2$-ECP over $\fqm$ for $C$.
So $A$ has parameters $[n,3,n-2]$, by \cite[Proposition 2.5]{pellikaan:1996}, and $d(B^\perp)\geq 3$.
After a possible extension of the field $\fqm$ there exists a subcode $B_0$ of $B$ with parameters $[n,2,n-1]$ such that $(A,B_0)$ is a $2$-ECP for $C$ by \cite[Corollary 5.4]{pellikaan:1996}.
And, after another field extension we may assume that $B_0$ has a codeword of weight $n$ and that it is the all-ones vector. 
Moreover, $B_0$ is an MDS code so $B_0$ has a generator matrix $G_{B} = \left(\begin{array}{c|c} I_2 & B \end{array}\right)$, where $B$ is a $2\times (n-2)$ matrix with entries $b_{ij}$ in $\mathbb F_{q^m}$. Let $\mathbf x = \left( \begin{array}{ccccc}1 & 1 & b_{13} & \ldots & b_{1n}\end{array}\right)$, by replacing $B_0$ by $\mathbf x^{-1} * B_0$ we may assume that $b_{1j} = 1$ for all $3\leq j \leq n$. Thus, $B_0$ is generated by $\left( \begin{array}{ccccc} 1 & 0 & 1 & \ldots & 1\end{array}\right)$ and $\mathbf b = \left(\begin{array}{ccccc} 0 & 1 & b_{23} & \ldots & b_{2n}\end{array}\right)$. Take notice that the entries of $\mathbf b$ are mutually distinct since $B_0$ is MDS, or equivalently, since all determinants of $2\times 2$ minors of $B$ are nonzero.
Clearly $B_0$ is the code $\mathrm{GRS}_2(\bb, \mathbf 1)$.

Let $H_C$ be the parity check matrix of $C$ in reduced echelon normal form.
Then, $H=\left(\begin{array}{c|c}I_4&P\end{array}\right)$ where $P$ is a $4\times (n-4)$ matrix with entries $p_{i,j}$ with $1\leq i \leq 4$ and $1\leq j \leq n-4$.
Let $\mathbf x = \left(\begin{array}{ccccccc}0 & 0 & 0 & 1 & p_{4,1} & \ldots & p_{4,n-4}\end{array}\right)$, by replacing $C$ by $x*C$ we may assume that the $4$-th row of $P$ is the all-ones vector.
Note that $G_C = \left(\begin{array}{c|c} -P^T & I_{n-4}\end{array}\right)$ is a generator matrix of $C$.

Let $G_A$ be a generator matrix of $A$ in reduced echelon form:
$$G_A = \left(\begin{array}{c|c|c} I_3 & \begin{array}{c}r_1 \\ r_2 \\ r_3 \end{array} & S \end{array}\right)$$
where $S$ is a $3\times (n-4)$ matrix with entries $s_{i,j}$ with $1\leq i \leq 3$ and $1\leq j \leq n-4$.

Then, $A\subseteq (C^\perp)\otimes \fqm$, since $\mathbf 1 \in B$ and $(A*B) \perp C$.
Since $G_C$ is a parity check matrix of $(C^\perp)\otimes \fqm$. Then, 
the inner product of any row of the matrix $G_A$ and any column of the matrix $G_C^T$ is zero.
Hence,
$$
s_{ij}=p_{ij}+r_i \ \mbox{ for all }\ i=1,2,3 \ \mbox{ and } \ 1 \leq j \leq n-4.
$$
Also $A* \bb \subseteq (C^\perp)\otimes \fqm$, since $\bb \in B$ and $(A*B) \perp C$.
So
$$
s_{ij}b_{j+4}=b_ip_{ij}+r_ib_4 \ \mbox{ for all } \ i=1,2,3 \ \mbox{ and } \ 1 \leq j \leq n-4.
$$
Therefore $(p_{ij}+r_i)b_{j+4}=b_ip_{ij}+r_ib_4$.  Thus,
$$
p_{ij} = \frac{r_i(b_4-b_{j+4})}{b_{j+4}-b_i} \ \mbox{ for all } \ i=1,2,3 \ \mbox{ and } \ 1 \leq j \leq n-4.
$$
Let $\bc$ be the $n$-tuple with entries $c_i=r_i^{-1}$ for $i=1,2,3$, $c_4=-1$ and
$c_{j+4}=b_4-b_{j+4}$ for $1 \leq j \leq n-4$.
Then,
$$
p_{ij} = \frac{c_{j+4}}{c_i[b_{j+4},b_i]} \ \mbox{ for all } \ i=1,2,3,4 \ \mbox{ and } \ 1 \leq j \leq n-4.
$$
Hence $(C^\perp)\otimes \fqm$ is the generalized Cauchy code $C_4(\bb,\bc)$.
So $C$ is a GRS code defined over $\fqm$,  but also over $\fq $ by Proposition \ref{pGRSdefoverFq}.

We have already seen that $B_0 = \mathrm{GRS}_2(\mathbf b, \mathbf 1)$.
Let $\mathbf d$ be the $n$-tuple with entries $d_i = r_i^{-1}$ for $i=1, 2,3$ and $d_{i+3} = 1$ for $1\leq i \leq n-3$. Then, $A$ has a generator matrix of the form $\left( \begin{array}{c|c} I_3 & Q\end{array}\right)$ with 
$$q_{ij} = \frac{d_{j+3}}{d_i[b_{j+3},b_i]} \ \mbox{ for all } \ i=1,2,3 \ \mbox{ and } \ 1 \leq j \leq n-3.$$

Hence $A$ is the generalized Cauchy code $C_3(\bb,\bd)$, so also the code $A$ is a GRS code with the same same evaluation-point sequence $\bb$.

Therefore we actually have that $A=GRS_3(\bb,\be)$ and $C=GRS_{n-4}(\bb,\mathbf f)$ for some nonzero $n$-tuples $\be$ and $\mathbf f$.
Hence $A*C=GRS_{n-2}(\bb, \be*\mathbf f)$.
Recall that $(A*B)\perp C$, so $(A*C)\perp B$. Hence $B_0\subseteq B\subseteq (A*C)^\perp$,
whereas $B_0$ and $(A*C)^\perp$ both have dimension $2$.
Therefore, $B_0=B=(A*C)^\perp$ and $B$ is the code $GRS_2(\bb, \mathbf 1)$
with the same evaluation-point sequence $\bb$ as $A$ and $C$.
\end{proof}

In the following theorem we generalize this result by proving that every MDS code with minimum distance $2t+1$ having a $t$-ECP is a generalized Reed-Solomon code. Hence if there exists linear codes not admitting an error correcting pair then they belongs to the family of MDS codes and its subcodes excluding the class of GRS.

\begin{thm}\label{mainthm}Let $t$ be an integer such that $1\leq t \leq n/2-1$.

Let $C$ be an $[n, n-2t, 2t+1]$ code over $\mathbb F_q$ that has $(A,B)$ as a $t$-ECP over a finite extension of $\fq$.
Then, $C$ is a GRS code defined over $\fq$. 
Moreover, $A,B$ are GRS codes with the same evaluation-point sequence  as $C$
and $B=(A*C)^\perp$.
\end{thm}

\begin{proof}
By the assumptions we have that $C$ has dimension $k=n-2t$ and $1\leq t \leq n/2-1$. Hence $2\leq k \leq n-2$.
We proceed by induction on $t$.
The case $t=2$ was proved in \cite[Theorem 6.5]{pellikaan:1996} and Proposition \ref{t=2}.
For the inductive step, let $t>2$ and assume that the theorem holds for all $t'< t$.

Now suppose that $C$ is a code with parameters $[n,n-2t,2t+1]$ that has $(A,B)$ as $t$-ECP.
$A$ has parameters $[n,t+1,n-t]$ by \cite[Proposition 2.5]{pellikaan:1996}.

Let $C_1$ be the code obtained from $C$ by puncturing on the two right-most coordinates, since $n-2 \geq k$, then $C_1$ is a MDS code with parameters $[n-2,n-2t,2t-1]$ by Lemma \ref{MT:lemma}.

Let $A_1$ be the code obtained from $A$ by puncturing on the right-most coordinate and then shortening this punctured code on the right, that is:
$$
A_1 = \left\{ \mathbf a' = (a_1, \ldots, a_{n-2})\in \mathbb F_q^{n-2} \mid (\mathbf a', 0 , a_n)\in A \hbox{ for some } a_n \in \mathbb F_q\right\}.
$$
First we deduce that the code $A_{\{n\}}$ which is obtained from $A$ by puncturing on the right-most coordinate is a MDS code with parameters $[n-1,t+1,n-t-1]$, by Lemma \ref{MT:lemma}. Therefore, again using Lemma \ref{MT:lemma}, $A_1$ is a MDS code with parameters $[n-2,t,n-t-1]$.

Let $B_1$ be the code obtained from $B$ by shortening on the right-most coordinate and then puncturing this shortened code on the right that is
$$
B_1 = \left\{ \mathbf b' = (b_1, \ldots, b_{n-2}) \in \mathbb F_q^{n-2} \mid (\mathbf b', b_{n-1}, 0)\in B \hbox{ for some } b_{n-1}\in \mathbb F_q\right\}.
$$
Furthermore:
\begin{enumerate}
\item $A* B \perp C$ thus for all $\mathbf a' \in A_1$, $\mathbf b' \in B_1$ and $\mathbf c' \in C_1$ there exists some $a_n, b_{n-1}, c_{n-1}, c_n\in \mathbb F_q$ such that $\mathbf a = (\mathbf a',0,a_n)\in A$, $(\mathbf b', b_{n-1}, 0)\in B$ and $(\mathbf c', c_{n-1}, c_n)\in C$ verifying that
$$(\mathbf a' * \mathbf b') \cdot \mathbf c' = \sum_{i=1}^{n-2} a_i b_i c_i =
\sum_{i=1}^{n-2} a_i b_i c_i + 0 (b_{n-1}c_{n-1}) + 0(a_n c_n) = (\mathbf a * \mathbf b) \cdot \mathbf c = 0.$$
Hence $A_1 * B_1 \perp C_1$.
\item $k(A_1) = t > t-1$.
\item The code $(B_1)^\perp$ is obtained from $B^\perp$ by puncturing on the right-most coordinate and then shortening this punctured code on the right, by Section \ref{Puncturing:Section}.
Hence, $d((B_1)^\perp)> t-1$, since $d(B^\perp)> t$, and by puncturing the minimum distance drops at most by one, and by shortening the minimum distance remains the same.
\end{enumerate}
Thus $(A_1, B_1)$ is a $(t-1)$-ECP for $C_1$ which is a MDS code with $d(C_1) = 2(t-1) +1$. By the induction hypothesis we have that $C_1$ is a GRS code defined over $\fq$ and $A_1,B_1$ are GRS codes with the same evaluation-point sequence as $C_1$.

Following the same arguments we deduce that the code $C_2$ obtained from $C$ by puncturing on the two left-most coordinates is a GRS code defined over $\fq$.

Let Let $I=\{n-1,n \}$ and $J= \{1,2 \}$. Then, $C_1=C_I$ and $C_2=C_J$.
Hence $C$ is a GRS code by Proposition \ref{puncture}.
Furthermore $C$ is an $\fq$-linear code. Hence  $C$ is a GRS code defined over $\fq $ by Proposition \ref{pGRSdefoverFq}. 
More precisely, $C^\perp=GRS_{2t}(\br,\bs)$ with evaluation-point sequence $\br$ and some vector $\bs$ with nonzero entries.

We now show that $A$ and $B$ are GRS codes defined with the same evaluation-point sequence as $C$.

Without loss of generality we may assume that $\mathbf 1 \in B$. 
So $A \subseteq C^\perp$.
For any $\fqm$-linear code $V$ we define $V'= \{ \bx \in V | x_{n-1}=0\}$.
Then, 
$$A' \subseteq (C^\perp)' = \bu* \mathrm{GRS}_{2t-1}(\br,\bs) \mbox{ with } \bu = \br - r_{n-1}\mathbf 1.$$
Let $\bx" =(x_1,\ldots ,x_{n-2})$ for any vector $\bx \in \fqm^n$.
Then,
$$
A_1=(A')_I \subseteq (\bu* \mathrm{GRS}_{2t-1}(\br,\bs))_I =\mathrm{GRS}_{2t-1}(\br",\bs"*\bu").
$$
Recall that, by the Induction Hypothesis, $A_1$ is a GRS code of dimension $t$ and with evaluation-point sequence $\br''$. 
Hence $(A')_I = A_1= \mathrm{GRS}_{t}(\br",\bs"*\bu")$. But also, $A_1 = (\bu* \mathrm{GRS}_{t}(\br,\bs))_I$.
So 
$$A' =\bu* \mathrm{GRS}_{t}(\br,\bs)\mbox{, since }|I|=2< n-2t+1 = d(\bu* \mathrm{GRS}_{2t}(\br,\bs))$$ 
The dimensions of $A'$ and $A$ are $t$ and $t+1$, respectively.
Hence, 
$$A = \langle \mathbf x \rangle + \bu* \mathrm{GRS}_{t}(\br,\bs) \mbox{ for some }\bx.$$ 
Without loss of generality we may assume that $\bx $ has nonzero entries.

Now, $B* \bx \subseteq C^\perp$, or equivalently, $B\subseteq \bx^{-1} *C^\perp$.
Similarly as for $A$ one shows that 
$$B= \langle \mathbf 1 \rangle +\bv* \mathrm{GRS}_{t-1}(\br,\bs*\bx^{-1}) \mbox{, where }\bv = \br - r_n\mathbf 1.$$ 

Since $A_1*B_1 \subseteq C_1^{\perp} = \left(\mathrm{GRS}_{2t}(\br, \bs)\right)_I$ we have that:
$$\mathrm{GRS}_{t}(\br'',\bs'')*\mathrm{GRS}_{t-1}(\br'',\bs''*(\bx^{-1})'') = \mathrm{GRS}_{2t-1}(\br'',\bs''*\bs''*(\bx^{-1})'')\subseteq  \left(\mathrm{GRS}_{2t}(\br, \bs)\right)_I
$$
Thus, $\bx''=\lambda \bs''$. Moreover, $\bx, \bs\in \mathrm{GRS}_{2t}(\br, \bs)$, since $A\subseteq C^{\perp}$. So $\bx = \lambda \bs$, which gives that 
$$\begin{array}{ccc}
A=\mathrm{GRS}_{t+1}(\br,\bs)&\mbox{ and }&B=\mathrm{GRS}_{t}(\br,\mathbf 1).
\end{array}$$
Thus, $B=(A*C)^{\perp}$ is proved.
\end{proof}

\begin{cor}\label{unicity}\label{changed}Let $t$ be an integer such that $1\leq t \leq n/2-1$.

Let $C$ be an $[n, n-2t, 2t+1]$ code over $\mathbb F_q$ that has a $t$-error correcting pair
$(A,B)$ over a finite extension of $\mathbb F_q$. Then, any other $t$-error correcting pair
$(A',B')$ over a finite extension of $\mathbb F_q$ is such that $A'=\bx *A$ and $B'=\bx^{-1}*B$
where $\bx $ has nonzero entries in a common  finite extension of $\mathbb F_q$.
\end{cor}

\begin{proof}$A$ and $A'$ are GRS codes of dimension $t+1$ with the same
evaluation-point sequence, by Theorem \ref{mainthm}. Hence, 
$$\begin{array}{cccc}
A=\mathrm{GRS}_{t+1}(\ba,\bb) & \mbox{ and  }&
A'=\mathrm{GRS}_{t+1}(\ba,\bb'),&
\mbox{for some }\ba, \bb \mbox{ and }\bb'.
\end{array}$$
Let $\bx$ be the $n$-tuple with entries $x_j=b'_jb_j^{-1}$. Then,
$$A'=GRS_{t+1}(\ba,\bb')=\bx *GRS_{t+1}(\ba,\bb)=\bx *A.$$
Furthermore, $B'=\bx^{-1}*B$, since $B=(A*C)^\perp$ and $B'=(A'*C)^\perp$ by Theorem \ref{mainthm}.
\end{proof}

\section{Second proof}
\label{Second:Section}

A second  independent proof of the main theorem uses a recent result by Randriambololona
\cite{randriambololona:2013} and Mirandola-Z\'emor \cite{mirandola:2015} on the Schur product of codes.

The following proposition is called the Product Singleton Bound.

\begin{prop}\label{pmds1}
Let $(A,B)$ be a pair of $\fq $-linear codes in $\fq ^n$. Then,
$$
d(A*B) \leq \max \{ 1, n - ( \dim A +\dim B) +2 \}.
$$
\end{prop}

\begin{proof}\cite[Proposition 5]{randriambololona:2013}
\end{proof}

A pair $(A,B)$ of $\fq $-linear codes in $\fq ^n$ is called a {\em Product MDS} (PMDS) pair if
equality holds in the above proposition.

The stabilizer of an $\fq$-linear  code $C$ in $\fq^n$ is defined by
$$
\mbox{St(C)} = \{ \bx \in \fq^n | \bx * C \subseteq C \} .
$$
The following inequality is called the Kneser bound.
\begin{prop}\label{kneser1}
Let $(A,B)$ be a pair of $\fq $-linear codes in $\fq ^n$. Then,
$$
\dim A*B  \geq  \dim A + \dim B - \dim \mbox{St(A*B)}.
$$
\end{prop}

\begin{proof}\cite[Theorem 3.3]{mirandola:2015}.
\end{proof}

\begin{prop}\label{pmds2}
Let $(A,B)$ be a PMDS pair of $\fq $-linear codes in $\fq ^n$.
Further assume that $n> \dim A + \dim B$.
Then, $A$, $B$ and $A*B$ are MDS codes.
\end{prop}

\begin{proof}\cite[Proposition 5.5]{mirandola:2015}.
\end{proof}

\begin{prop}\label{pmds3}
Let $(A,B)$ be a PMDS pair of $\fq $-linear codes in $\fq ^n$.
Further assume that $n> \dim A + \dim B$ and $\dim A, \dim B \geq 2$.
Then, $A$, $B$ and $A*B$ are GRS codes with a common evaluation-point sequence.
\end{prop}

\begin{proof}\cite[Corollary 5.6]{mirandola:2015}
\end{proof}

\begin{thm}[Second Proof for Theorem \ref{mainthm}]
\label{mainthm2}

Let $t$ be an integer such that $2< t < \frac{n}{2}-1$. 
Let $C$ is an $[n, n-2t, 2t+1]$ code over $\mathbb F_q$ that has $(A,B)$ as a $t$-ECP over a finite extension of $\mathbb F_q$.
Then, $C$ is  GRS code defined over $\mathbb F_q$. Moreover, $A$, $B$ are GRS codes with the same evaluation-point sequence as $C$ and $B=(A*C)^{\perp}$.
\end{thm}

\begin{proof}
Let $C$ be an $[n, n-2t, 2t+1]$ code over $\fq$ that has a $t$-error correcting pair
$(A,B)$ over a finite extension of $\fq$.
Then, as argued before in the proof of Theorem \ref{mainthm}, $A$ has parameters $[n,t,n-t+1]$ and, after a possible extension of the field $\mathbb F_{q^m}$, there exists a subcode $B_0$ of $B$ with parameters $[n,t,n-t+1]$ such that $(A,B_0)$ is a $t$-ECP for $C$. 
Take notice that $\dim B \geq \dim B_0$ and $d(B) \leq d(B_0)$. 

Now $A*B \subseteq C^\perp$
and $d(C^\perp )= n-2t+1$, since $C^\perp$ is also MDS. Therefore
$$
d(A*B) \geq d(C^\perp ) = n-2t+1 = n - ( \dim A +\dim B) +2.
$$
Furthermore $n-2t+1 >1$, since $d=2t+1\leq n$. Hence $(A,B)$ is a PMDS pair by Proposition \ref{pmds1}.

Now $\dim A =t+1$, $\dim B=t\geq 2$ and $n > \dim A + \dim B$ by assumption.
So $A*B$ is MDS of dimension $n-d(A*B)+1 = 2t$, by Proposition \ref{pmds2}.
Hence $A*B=C^\perp$.
So $A$, $B$ and $C^{\perp} = A*B$ are GRS codes with a common evaluation-point sequence by Proposition \ref{pmds3}.
\end{proof}

\section{Conclusion}

The class of MDS codes of minimum distance $2t+1$ that have a $t$ error correcting pair coincides with the class of generalized Reed-Solomon codes.
Two proofs are given, the second one uses a recent result \cite{mirandola:2015} on the Schur product of codes.

Since GRS codes can be viewed as algebraic geometry (AG) codes on the projective line, that is the curve of genus zero
and AG codes have ECP's, one might ask whether there is a converse such as the Main Theorem but now for AG codes of genus larger than zero.

\section*{Acknowledgement}
This work arose from the discussion with Naomi Benger, Julia Borghoff and the authors in the Research Retreat of
the {\em Code-Based Cryptography Workshop}, May 2011 at the Technical University of Eindhoven.
It was presented by the first author at the {\em Code-Based Cryptography Workshop}, May 2012 at the Technical University of Denmark, Lyngby and also at the conference {\em Applications of Computer Algebra}, July 2015 at Kalamata.

\bibliographystyle{plain}
\bibliography{ECP-MDS}
\end{document}